\documentclass[a4paper,11pt]{article}

\usepackage{amssymb,latexsym,amsthm,amsmath,stmaryrd,amsfonts}
\usepackage{mathabx}
\usepackage{graphicx}
\usepackage{color}%
\usepackage{times}

\newcommand{\R}{\mathbb{R}}

\newtheorem{thm}{Theorem}[section]
\newtheorem{lem}[thm]{Lemma}
\newtheorem{prop}[thm]{Proposition}

\theoremstyle{definition}
\newtheorem{defn}[thm]{Definition}

\theoremstyle{remark}

\newcommand{\tos}{\rightrightarrows} 

\usepackage{hyperref}

\title{Time-Dependent Generalized Nash Equilibrium Problem}
\author{John Cotrina
	\thanks{Universidad del Pac\'ifico. Av. Salaverry 2020, Jes\'us Mar\'ia, Lima, Per\'u. Email: 
	\texttt{\{ cotrina\_je,~zuniga\_jj\}@up.edu.pe}} 
	\and Javier Z\'u\~niga\footnotemark[1]}

\begin{document}
\maketitle
\begin{abstract}
We prove an existence result for the time-dependent generalized Nash equilibrium problem under generalized convexity using 
a fixed point theorem.
Furthermore, an application to the dynamic abstract economy is considered.
\end{abstract}
\bigskip

\bigskip
\section{Introduction and definitions}
Here, we consider the Lebesgue space $L^2([0,T],\R^n)$ with the inner product
\[
 \langle\!\langle \phi,\psi\rangle\!\rangle=\int_0^T\!\phi(t)\psi(t) dt.
\]
The time-dependent generalized Nash equilibrium problem (time-dependent GNEP) is a 
generalized Nash game in which each player's strategy and objective function depend on time.
More precisely, let us assume that we have $p$ players and to each player $\nu \in \{ 1, 2, ..., p \}$ we can associate a natural number 
$n_\nu$. Set $\displaystyle n = \sum_{\nu=1}^p n_\nu$. Each player has a strategy
$x^\nu \in X_{\nu}(x^{-\nu})\subset L^2([0,T], \R^{n_\nu})$, where by $x^{-\nu} \in L^2([0,T], \R^{n - n_\nu})$ 
we mean to denote the vector formed by all players' strategies except for those of player $\nu$. The set $X_{\nu}(x^{-\nu})$ 
is the strategy space of player $\nu$ given the strategy of the other players as it is usually the case in non-cooperative games.
We can also write $x = (x^\nu, x^{-\nu}) \in \displaystyle\prod_{\nu=1}^p X_{\nu}(x^{\nu})\subset L^2([0,T], \R^n)$ 
 which is shorthand (already used in many papers on the subject, see e.g. \cite{AGM,Facchinei-Kanzow})
to denote $x = (x^1, ..., x^{\nu-1}, x^\nu ,x^{\nu+1},..., x^p)$ as a way to single out the strategy of player $\nu$ within the full strategy vector.
Each $x^\nu(t) \in \R^{n_\nu}$ can be 
thought of as a strategy of player $\nu$ at time $t \in [0,T]$. Then $x \in L^2([0,T], \R^n)$ is the full strategy vector and thus $x(t)$ is the vector of strategies of all players at a given time $t \in [0,T]$.

Let $\theta_\nu : L^2([0,T], \R^n) \to \R$ be the objective function for player $\nu$. 
A strategy $\hat{x} \in K \subset L^2([0,T],\R^n)$ is a \emph{time-dependent generalized Nash equilibrium} (\cite{AGM})
if and only if for each player $\nu$, we have $\hat{x}^\nu \in X_\nu(\hat{x}^{-\nu})$ and 
\[ \theta_\nu (\hat{x}) \geq \theta_\nu (x^\nu, \hat{x}^{-\nu}), \quad \text{for all} \,\, x^{\nu} \in X_\nu(\hat{x}^{-\nu}). \]

In other words this means that $\hat{x}\in K\subset L^2([0,T],\R^{n_\nu})$ is a time-dependent generalized Nash equilibrium if
for all $\nu$, $\hat{x}^\nu\in L^2([0,T],\R^{n_\nu})$ solves the following optimization problem
\begin{equation*}
 \max_{x^\nu\in X_{\nu}(\hat{x}^{-\nu})} \theta_\nu(x^\nu,\hat{x}^{-\nu})
\end{equation*}

GNEPs with infinite-dimensional strategy spaces have been
investigated in recent years (see for instance \cite{AGM,CG,CY}). In \cite{CG,CY} 
the concavity or convexity of all objective functions was considered in order to obtain an existence result. Recently in \cite{AGM}
the authors give an existence result under generalized convexity (more precisely semistrictly quasiconcavity) and in the jointly convex case.
In this sense, we propose an existence result under only quasiconcavity of all
objective functions and classical hypothesis on the constraint set-valued map, which generalizes Theorem 6 in \cite{Facchinei-Kanzow}.

A function $f:L^2([0,T],\R^n)\to\R$ is said to be
\begin{itemize}
 \item \emph{quasiconcave} if for any $x,y\in L^2([0,T],\R^n)$ and $\lambda\in[0,1]$, we have 
\[
 f(\lambda x+(1-\lambda)y)\geq \min\{f(x),f(y)\};
\]
\item \emph{semistrictly quasiconcave} if it is quasiconcave and for any $x,y\in L^2([0,T],\R^n)$, such that $f(x)\neq f(y)$, and $\lambda\in ]0,1[$, we have 
\[
 f(\lambda x+(1-\lambda)y)> \min\{f(x),f(y)\}.
\]
\end{itemize}

A set-valued map $F:L^2([0,T],\R^n)\tos L^2([0,T],\R^n)$ is said to be:
\begin{itemize}
 \item \emph{upper semicontinuous} at the point $x\in L^2([0,T],\R^n)$ if for any open 
 $W$ such that $F(x)\subset W$, there exists a neighborhood $V$ of $x$ such that, for all $z\in V$, we have $F(z)\subset W$.
 \item \emph{lower semicontinuous} at the point $x\in L^2([0,T],\R^n)$ if for any open 
 $W$ such that $F(x)\cap W\neq\emptyset$, there exists a neighborhood $V$ of $x$ such that, for all $z\in V$, we have $F(z)\cap W\neq\emptyset$.
\end{itemize}

Our existence result will be obtained as a consequence of Kakutani's Theorem which is stated in the next result and it can be found in \cite{GD}.
\begin{thm}[Kakutani]\label{Kakutani}
Let $K$ be a nonempty compact convex subset of a locally convex space $E$ and 
let $T: K \tos K$ be a set-valued map. If $T$ is upper semicontinuous such that for all $x\in K$, $T(x)$ is nonempty, closed and convex, 
then $T$ admits a fixed point.
\end{thm}


\section{Existence Result}

As we have already mentioned earlier, our aim is to prove the existence of a time-dependent Nash equilibrium, 
and this will be done thanks to a reformulation of the equilibrium problem into an associated fixed point problem.

With the previous notation of time-dependent GNEP, for any given optimal strategy $x^{-\nu}$ of the rival players, let us define 
\[ S_\nu (x^{-\nu}) = \{ \hat{x}^\nu \in X_\nu(x^{-\nu}): \,\, \theta_\nu (\hat{x}^\nu,x^{-\nu}) \geq \theta_\nu (x^\nu, x^{-\nu})
~\text{for all} \,\, x^{\nu} \in X_\nu(x^{-\nu}) \}, \] 
and also the set-valued map $ S:K\tos K$ defined as
$\displaystyle S(x) = \prod_{\nu=1}^p S_\nu(x^{-\nu})$, where $ K = \displaystyle\prod_{\nu=1}^p K_\nu $.

The following proposition connects the notions of equilibrium and fixed point.

\begin{prop} \label{equiv}
Let $\hat{x}\in K$, then $\hat{x}$ is a time-dependent generalized Nash equilibrium if and only if it is a fixed point of $S$.
\end{prop}
\begin{proof}
The vector $\hat{x}$ is a time-dependent generalized Nash equilibrium if and only if for each $\nu$ it satisfies 
$ \theta_\nu(\hat{x}^\nu,\hat{x}^{-\nu})\geq  \theta_\nu(x^\nu,\hat{x}^{-\nu})$, for all $x^\nu\in X_\nu(\hat{x}^{-\nu})$,
i.e. $\hat{x}^\nu\in S_\nu(\hat{x}^{-\nu})$, which is equivalent to $\hat{x}\in S(\hat{x})$.
\end{proof}

We are ready for our main result.
\begin{thm}\label{principal}
With the previous notation, let us assume that:
\begin{enumerate}
\item There exist $p$ nonempty, convex and weakly compact sets $K_\nu\subset L^2([0,T],\R^{n_\nu})$ such that
for every $x\in L^2([0,T],\R^{n})$ with $x^{\nu}\in K_\nu$ for every $\nu$, $X_{\nu}(x^{-\nu})$ is nonempty,
closed and convex, $X_\nu(x^{-\nu})\subset K_\nu$, and $X_\nu,$ as a set-valued map, is both upper and lower semicontinuous;
\item the function $\theta_\nu$ is continuous for every player $\nu$;
\item the function $\theta_\nu(\cdot,x^{-\nu})$ is quasiconcave on $X_\nu(x^{-\nu})$ for every player $\nu$.
\end{enumerate}
Then there exists a time-dependent generalized Nash equilibrium.
\end{thm}
In order to proof the above theorem we need Berge's maximum theorem which can be found in \cite{EO}.
\begin{thm}[Berge's maximum theorem]\label{Berge}
Let X,Y be two metric spaces, $f: X \times Y \to \R$ be a function and $F: X \to 2^Y$ a set-valued map. Assume that $f$ is continuous, $F$ is both upper and lower semicontinuous; and $F$ is nonempty and compact valued. Then it follows that 
\begin{enumerate}
\item $\displaystyle \varphi : x \to \max_{y \in F(x)} f(x,y)$ is a continuous function from $X$ to $\R$.
\item $\displaystyle \Phi: x \to \arg \max_{y \in F(x)} f(x,y)$ is an upper semicontinuous set-valued map from $X$ to $2^Y$ and compact-valued.
\end{enumerate}
\end{thm}
\begin{proof}[Proof of Theorem \ref{principal}]
The continuity of $\theta_\nu$ and the upper and lower semicontinuity of $X_\nu$ implies, by Theorem \ref{Berge}, 
the upper semicontinuity of the set-valued map $S_\nu$, for all $\nu$.
Moreover,  $S_\nu$ has closed convex values due to conditions 2 and 3.
Therefore, the set-valued map $S$ is upper semicontinuous with closed convex values. 
Since $K$ is weakly compact, applying Theorem \ref{Kakutani}, $S$ has at least one fixed point. By Proposition~\ref{equiv} the result follows.
\end{proof}


\section{Application: Time-dependent Abstract Economy}
Arrow and Debreu in 1954 (see \cite{A-D}) considered a general ``economic system'', named \emph{abstract economy}, 
along with a corresponding definition of equilibrium.
After this pioneering work, several authors established the existence of an equilibrium that included production 
and consumption, see for instance \cite{B-D-M,D-M-V}. 

Recently Donato \emph{et al.} introduced in \cite{Donato-Milasi-Vitanza}
the concept of time-dependent abstract economy, for which they give an existence result by a variational reformulation.
Motivated by this last work, we give an existence solution for a time-dependent abstract economy problem using the 
time-dependent GNEP formulation. More precisely,
we suppose there are $l$ distinct commodities (including all kinds of services).
Each commodity can be bought or sold at a finite number of locations (in space and time).
The commodities are produced in ``production units'' (companies), whose number is $s$. For each production unit $j$
there is a set $A_j$ of possible production plans. An element $a^j\in A_j$ is a vector in $L^2([0,T],\R^l)$.
We note that the sign of the $h$th component at time $t$ of this last vector       has a particular meaning. When the quantity $a_h^j(t)$ is positive, it 
represents the commodity offered in the market by the production unit $j$ at time $t$, this is known as an output. When $a^j_h(t)$ is negative, it represents the amount of this commodity that will be used in the production process (like raw materials), this is known as an input. When $a^j_h(t)$ equals to zero the production unit $j$ does not produce the commodity $h$ at time $t$ nor it is required in the production process (it is not an input nor and output).
If we denote by $p\in L^2([0,T],\R^l)$ the prices of the commodities, the production units will naturally aim at maximizing the total 
revenue.

We also assume the existence of ``consumption units'', typically families or individuals, whose number is $r$. Associated to each 
consumption unit $i$ we have a vector $b^i\in L^2([0,T],\R^l)$, where $b_h^i(t)$ represents the quantity of the $h$th commodity
consumed by the $i$th individual at time $t$. When $b_h^i(t)$ is positive, it represents the amount of commodity $h$ being consumed in the market by the consumption unit $i$ at time $t$. When it is negative, it represents a labor service being offered by the consumption unit $i$ at time $t$. When this quantity is zero the consumption unit $i$ does not consume the commodity $h$ at time $t$ nor it offers it as a labor service. In general, $b^i$ must belong to a certain set $B_i\subset L^2([0,T],\R^l)$ which is convex, closed and bounded from below, i.e., there is a vector $\beta^i$ such that $\beta^i_h(t) \le b^i_h(t)$ a.e. in $[0,T]$, for all $h$. The set $B_i$ includes all consumption
vectors among which the individual could choose one if there were no budgetary constraints (the latter constraints will be explicitly formulated below).
We also assume that the $i$th consumption unit is endowed with a vector $\xi^i\in L^2([0,T],\R^l_+)$ of initial holding of commodities and has a contractual
claim to the share $\alpha_{ij}$ of the profit of the $j$th production unit such that $\alpha_{ij}\geq0$ and
$\displaystyle\sum_{i=1}^r\alpha_{ij}=1$ for all $j$. Under these conditions it is then clear that, given 
a vector of prices $p$, the choice of the $i$th unit is further restricted to those vectors $b^i\in B_i$ such that
\[
\langle\! \langle p,b^i\rangle\!\rangle\leq \langle\!\langle p,\xi^i\rangle\!\rangle+\sum_{j=1}^s\alpha_{ij}\langle\!\langle p,a^j\rangle\!\rangle
\]
As it is standard in economic theory, the consumption units aim is to maximize a utility function $u_i(b^i)$.

\begin{defn}
A \emph{time-dependent economic equilibrium} is a vector of the form
\[
(\hat{a}^1,\cdots,\hat{a}^s,\hat{b}^1,\cdots,\hat{b}^r,\hat{p})
\]
such that
\begin{eqnarray}\label{PP}
\langle\!\langle \hat{p},\hat{a}^j\rangle\!\rangle &=& \max_{a^j\in A_j} \langle\!\langle \hat{p},a^j\rangle\!\rangle,\mbox{ for all }j\\\label{CP}
u_i(\hat{b}^i) &=& \max_{b^i\in D_i(\hat{a},\hat{p})} u_i(b^i),\mbox{ for all }i\\\label{MP}
\sum_{i=1}^r\langle\!\langle \hat{p}, \hat{b}^i-\xi^i\rangle\!\rangle-
\sum_{j=1}^s\langle\!\langle \hat{p}, \hat{a}^j\rangle\!\rangle &=&
\max_{p\in P}\sum_{i=1}^r\langle\!\langle p, \hat{b}^i-\xi^i\rangle\!\rangle-
\sum_{j=1}^s\langle\!\langle p, \hat{a}^j\rangle\!\rangle
\end{eqnarray}
where $
D_i(a,p)=\displaystyle\left\lbrace b^i\in B_i:~ \langle\!\langle p,b^i\rangle\!\rangle\leq\langle\!\langle p,\xi^i\rangle\!\rangle+
\max \left\lbrace 0,\sum_{j=1}^s\alpha_{ij}\langle\!\langle p,a^j\rangle\!\rangle\right\rbrace \right\rbrace$ 
and $P=\displaystyle\left\lbrace p\in L([0,T],\R^l):~p\geq0,~ \frac{1}{T} \int_T \sum_{h=1}^l p^h(t)dt=1\right\rbrace$.

\end{defn}

The following lemma says that the consumption units' demand must be always satisfied in average by the production units. Moreover,
the commodities whose price is zero  a.e. in $[0,T]$, are only possible if the supply exceeds the demand a.e. in $[0,T]$.

\begin{lem}
 Let $A_j$ be a set such that $0\in A_j$, for all $j$. If $(\hat{a},\hat{b},\hat{p})$ is a time-dependent economic equilibrium then 
 \begin{equation}\label{int-0}
 \int_0^T\left[\sum_{i=1}^r (\hat{b}_h^i(t)-\xi_h^i(t))-\sum_{j=1}^s\hat{a}_h^j(t)\right]dt\leq0,~\mbox{for all }h
\end{equation}
Moreover, if there exists a consumption unit $i_0$ whose utility function $u_{i_0}$ has no maximum in $B_{i_0}$ (non-satiation), i.e., for all $b^{i_0} \in B_{i_0}$ there exists a $\overline{b}^{i_0} \in B_{i_0}$ such that $u_{i_0}(\overline{b}^{i_0})>u_{i_0}(b^{i_0})$,
then
\begin{equation}\label{orto}
\sum_{i=1}^r\langle\!\langle \hat{p}, \hat{b}^i-\xi^i\rangle\!\rangle-\sum_{j=1}^s\langle\!\langle \hat{p},a^j\rangle\!\rangle =0.
\end{equation}
\end{lem}

\begin{proof}
This result follows from some arguments that appear in Theorem 3 of \cite{D-M-V} and Section 1.4.2 of \cite{A-D}.
Since $0\in A_j$ and (\ref{PP}) is satisfied,  $\langle\!\langle \hat{p},\hat{a}^j\rangle\!\rangle\geq0$ and thus, \[ \sum_{j=1}^s\langle\!\langle \hat{p}, \alpha_{ij}\hat{a}^j\rangle\!\rangle\geq0. \]
On the other hand, as $\hat{b}\in D(\hat{a},\hat{p})$, we deduce that
\begin{equation}\label{Dsat}
\displaystyle \sum_{i=1}^r\langle\!\langle \hat{p}, \hat{b}^i-\xi^i\rangle\!\rangle-
 \sum_{j=1}^s\langle\!\langle \hat{p},\hat{a}^j\rangle\!\rangle\leq 0 
\end{equation}
Now consider $p$ defined as
\[
 p^h=\left\lbrace\begin{array}{cc}
                  0,~h\neq h_0\\
                  1,~h=h_0
                 \end{array}\right.
\]
clearly $p\in P$. Therefore, by (\ref{MP}) we have
\[
 \displaystyle \left\langle\!\!\left\langle p,\sum_{i=1}^r (\hat{b}^i-\xi^i)-\sum_{j=1}^s\hat{a}^j\right\rangle\!\!\right\rangle=
 \int_0^T\left[\sum_{i=1}^r (\hat{b}_{h_0}^{i}(t)-\xi_{h_0}^{i}(t))-\sum_{j=1}^s\hat{a}_{h_0}^{j}(t)\right]dt\leq0.
\]

Moreover, if there exists a consumption unit $i_0$ whose utility function $u_{i_0}$ has no maximum in $B_{i_0}$, then for $\hat{b}^{i_0}$ there exists $b^{i_0}$ such that $u_{i_0}(b^{i_0})>u_{i_0}(\hat{b}^{i_0})$. This implies in turn that
$u_{i_0}(\lambda b^{i_0}+(1-\lambda)\hat{b}^{i_0})>u_{i_0}(\hat{b}^{i_0})$, for all $\lambda\in]0,1[$.
Suppose the strict inequality held in \eqref{Dsat}, i.e.
\[
\sum_{i=1}^r\langle\!\langle \hat{p}, \hat{b}^i-\xi^i\rangle\!\rangle-\sum_{j=1}^s\langle\!\langle \hat{p},a^j\rangle\!\rangle <0. 
\]
If this is the case we could choose a $\lambda$ small enough for which the inequality is still satisfied and 
$\lambda b^{i_0}+(1-\lambda)\hat{b}^{i_0}\in D_{i_0}(\hat{a},\hat{p})$ which is a contradiction. 
\end{proof}

Finally, we can prove the existence of a time-dependent economic equilibrium using Theorem~\ref{principal}.

\begin{thm}
If the following hold:
\begin{enumerate}
 \item[$(i)$] The set $A_j$ is convex and weakly compact and $0\in A_j$, for all $j$;
\item[$(ii)$] the function  $u_i$ is semistrictly quasiconcave and continuous, for all $i$.
\end{enumerate}
Then, there exists a time-dependent economic equilibrium.
\end{thm}
\begin{proof}
First, we consider $S=\prod S_h$  where
\[
 S_h=\left\lbrace f\in L^2([0,T],\R):~\begin{array}{l}
                                      f\geq0\mbox{ a.e. in }[0,T]\mbox{ and }\\
                                     \displaystyle \int_0^T\!f(t)dt\leq\int_0^T\sum_{i=1}^r\xi^i_h(t)dt+R
                                     \end{array}  \right\rbrace
\]
with $R>0$ and $\displaystyle\left| \int_0^T\sum_{j=1}^sa_h^j(t)dt\right|<R$, for all $h$.

We use the following notation:
\begin{itemize}
 \item $x=(a,b,p)=(a^1,\cdots,a^s,b^1,\cdots,b^r,p)$ and 
 \[
x^\nu=\left\lbrace\begin{array}{cl}
                   a^\nu,& \nu\in\{1,\cdots,s\}\\
                   b^{\nu-s},& \nu\in\{s+1,\cdots, s+t\}\\
                   p,&\nu=s+t+1
                  \end{array}
  \right.
 \]
\item The objective functions $\theta_\nu$ are defined as:
\[
 \theta_\nu(x)=\left\lbrace\begin{array}{cl}
                         \langle\langle p,a^\nu\rangle\rangle,& \nu\in\{1,\cdots, s\}\\
                         u_{\nu-s}(b^{\nu-s}),&\nu\in\{s+1,\cdots, s+t\}\\
                         \displaystyle\left\langle\!\!\left\langle p,\sum_{i=1}^t( b^i-\xi^i)-\sum_{j=1}^sa^j\right\rangle\!\!\right\rangle,&\nu=s+t+1
                        \end{array}
\right.
\]
Clearly, each objetive function is semistrictly quasiconcave.
\item The strategy set of player $\nu$ is defined as:
\[
X_\nu(x^{-\nu})=\left\lbrace
\begin{array}{cl}
A_\nu&\nu\in\{1,\cdots,s\}\\
D_{\nu-s}(a,p)\cap S,&\nu\in\{s+1,\cdots,s+t\}\\
P,&\nu=s+t+1
\end{array}
\right.
\]
It is clear that for each $x^{-\nu}$ the strategy set $X_\nu(x^{-\nu})$ is convex and closed.
\end{itemize}

By Theorem \ref{principal} the time-dependent GNEP has a solution of the form $\hat{x}=(\hat{a},\hat{b},\hat{p})$. 

Finally, we will prove that $(\hat{a},\hat{b},\hat{p})$ is a time-dependent equilibrium. First, notice that it holds (\ref{int-0}). Thus,
\[
\int_0^T\!\!\hat{b}_h^i(t)dt\leq  \int_0^T\!\!\sum_{i=1}^r\hat{b}_h^i(t)dt\leq \int_0^T\!\!\sum_{i=1}^r\xi_h^i(t)dt+ 
\int_0^T\!\!\sum_{j=1}^s\hat{a}_h^j(t)dt< \int_0^T\!\!\sum_{i=1}^r\xi_h^i(t)dt+R
\]
If there is a $\overline{b}^i\in D_i(\hat{a},\hat{p})$ such that $u_i(\overline{b}^i)>u_h(\hat{b}^i)$, then the semistrictly
quasiconcavity implies $u_i(\lambda \overline{b}^i+(1-\lambda)\hat{b}^i)>u_i(\hat{b}^i)$, for every 
$\lambda\in]0,1[$. Now,
consider
\[
J=\max_h\left\lbrace \left| \int_0^T\!\! \overline{b}_h^i(t)dt\right| +1\right\rbrace\mbox{ and }
I=\min_h\left\lbrace \int_0^T \!\sum_{i=1}^r\xi_h^i(t)dt+R-\int_0^T\!\! \hat{b}_h^i(t)dt\right\rbrace
\]
For $0<\lambda<\min\{1,I/J\}$, the following holds:
\[
\int_0^t\!\!(\lambda \overline{b}_h^i(t)+(1-\lambda)\hat{b}_h^i(t))dt\leq
\lambda\int_0^T\!\!\overline{b}_h^i(t)dt+\int_0^T\!\!\hat{b}_h^i(t)dt<
\int_0^T\!\!\sum_{i=1}^r\xi_h^i(t)dt+R
\]
that means, $\lambda \overline{b}^i+(1-\lambda)\hat{b}^i\in D_i(\hat{a},\hat{p})\cap S$. However, this is a contradiction.
\end{proof}


\begin{thebibliography}{1}
\bibitem{A-D}
K. J. Arrow and G. Debreu,
\textit{Existence of an equilibrium for a competitive economy.} Econometrica \textbf{22} (1954), 265--290.
%
\bibitem{AGM}
D. Aussel, R. Gupta, and A. Mehra,
\textit{Evolutionary Variational Inequiality Formualtion of the Generalized Nash Equilibrium Problem.} Journal of Optimization Theory and Applications \textbf{169} (2016), no. 1, 74--90. 
%
\bibitem{B-D-M} 
I. Benedetti, M. B. Donato, and M. Milasi,
\textit{Existence for Competitive Equilibrium by Means of Generalized Quasivariational Inequalities.} Abstract and Applied Analysis \textbf{2013} (2013), 8 pages.
%
\bibitem{CG} 
M. Castellani and M. Giuli,
\textit{An existence result for quasiequilibrium problems in separable Banach spaces.} Journal of Mathematical Analysis and Applications \textbf{425} (2015), no. 1, 85--95.
%
\bibitem{CY}
P. Cubiotti, and J.-C. Yao,
\textit{Nash equilibria of generalized games in normed spaces without upper semicontinuity.} Journal of Global Optimization \textbf{46} (2010), no. 4, 509--519.
%
\bibitem{D-M-V} 
M. B. Donato, M. Milasi and C. Vitanza,
\textit{Quasivariational Inequalities for Dynamic Competitive Economic Equilibrium Problem.} Journal of Inequalities and Applications \textbf{2009} (2009), no. 1, 519--623.
%
\bibitem{Donato-Milasi-Vitanza}
M. B. Donato, M. Milasi and C. Vitanza,
\textit{Evolutionary quasivariational inequality for a production economy.} Preprint (2017).
%
\bibitem{Facchinei-Kanzow}
F. Facchinei and C. Kanzow,
{\em Generalized Nash equilibrium problems}. 4OR \textbf{5} (2007), no. 3, 173--210.
%
\bibitem{GD}
A. Granas and J. Dugundji,
{\em Fixed point theory.} Springer-Verlag New York, 2003.
%
\bibitem{EO}
E. A. Ok,
{\em Real Analysis with Economic Applications.} Princeton University Press, 2005.

\end{thebibliography}
\end{document}